\documentclass[12pt,a4paper, reqno]{amsart}
\usepackage{amsfonts,amsmath,amsthm,amssymb, amscd}
\usepackage[top=2cm, bottom=5cm, left=3cm, right=1cm]{geometry}

\usepackage{color}

\def\GL{{\rm GL}}
\def\diag{{\rm diag}}

\usepackage{amsthm}
\newtheorem{theorem}{Theorem}
\newtheorem*{hypothesis} {Hypothesis}

\newtheorem{lemma}{Lemma}
\newtheorem{fact}{Fact}

\begin{document}

\title[Adequacy of nonsingular matrices]{ Adequacy of Nonsingular Matrices over Commutative Principal Ideal Domains}

\author{V.~Bovdi, V.~ Shchedryk}
\address[\texttt{V.~Bovdi}]
{United Arab Emirates University, UAE}
\email{vbovdi@gmail.com}
\address[\texttt{V.~ Shchedryk}]
{Pidstryhach Institute for Applied Problems of Mechanics and Mathematics,  National Academy of Sciences of Ukraine, Lviv,  Ukraine}
\email{shchedrykv@ukr.net}

\begin{abstract}
The notion of adequacy for commutative domains was introduced by Helmer in \textit{Bull. Amer. Math. Soc.}, \textbf{49} (1943), 225--236.
In the present paper, we extend the concept of adequacy to noncommutative B\'ezout rings.
We show that the set of nonsingular $2 \times 2$ matrices over a commutative principal ideal domain is adequate.
\end{abstract}
\maketitle

\section{Introduction and results}
Let $U(R)$ be the group of units of an associative,  commutative ring $R$ with $1\not=0$. The elements $a, b\in R$ are called  {\it strongly associated} if there exists  $e \in U(R)$ such that $a=be$ (see \cite[Definition 2.1, p.\,441]{Anderson_Valdes_Leon} and  \cite{Anderson_Valdes_Leon_2}).
The set of all non strongly  associate elements of the ring $R$ is denoted by $R^{*}$.  Of course,  we always assume   $1\in R^{*}$.
The matrix  $\diag(d_1,\ldots, d_n)$ means  a  matrix having $d_1,\ldots, d_n\in R$ on the  main diagonal and zeros elsewhere (by the main diagonal we mean the one beginning at the upper left corner). The set of all matrices of size $n\times m$ over a ring $R$ is denoted by $R^{n\times m}$.

A commutative ring $R$ is called an {\it elementary divisor ring}~\cite[p.\,465]{Kaplansky} if, for each matrix $A \in R^{n \times m}$, there exist invertible matrices $P_{A}$ and $Q_{A}$ such that
\begin{equation}\label{E:1}
P_{A} A Q_{A} = \diag(\alpha_{1}, \ldots, \alpha_{s}) \in R^{n \times m},
\end{equation}
where $s := \min(n, m)$ and each $\alpha_{i}$ divides $\alpha_{i+1}$ for $i = 1, \ldots, s - 1$.
The diagonal matrix $\diag(\alpha_{1}, \ldots, \alpha_{s})$ is called  a  {\it Smith form} of $A$ (unique up to strong associates of its diagonal elements).
Accordingly, we can always choose $\alpha_{1}, \ldots, \alpha_{s} \in R^{*}$ so that the matrix $\diag(\alpha_{1}, \ldots, \alpha_{s})$ is uniquely defined; it is called  the {\it Smith normal form} of the matrix $A$ and is denoted by $\mathrm{SNF}(A)$.
The matrices $P_{A}$ and $Q_{A}$ (see~\eqref{E:1}) are called the {\it left} and {\it right transforming} matrices of $A$, respectively.
The sets of all left and right transforming matrices of  $A \in R^{n \times n}$ with the Smith normal form  $\Phi:=\diag(\alpha_{1}, \ldots, \alpha_{n})$ have the form of  right and left  cosets $\mathbf{G}_\Phi P_A $ and $  Q_A\mathbf{G}_\Phi^T$ by the subgroups $\mathbf{G}_\Phi, \mathbf{G}_\Phi^T<\mathrm{GL}_n(R)$, respectively. Here   $\mathbf{G}_\Phi$ is the Zelisko group \cite{ZZ1, ZZ2, Zelisko} of the matrix $\Phi$, defined as
\[
\mathbf{G}_\Phi := \left\{ H \in \mathrm{GL}_n(R) \mid \exists S \in \mathrm{GL}_n(R) \text{ such that } H\Phi = \Phi S \right\}
\]
and $\mathbf{G}_\Phi^T:=\{H^T\mid H\in \mathbf{G}_\Phi\}$.


The {\it greatest common divisor} and the {\it least common multiple} of $a, b \in R$, which are unique up to strong associates, are denoted by $(a, b)$ and $[a, b]$, respectively; and $a \mid b$ means that $a$ is a divisor of $b$.

Let $R$ be a commutative  domain with $1\not=0$ in which every finitely generated ideal is principal (B\'ezout domain).
Let $a, b \in R$, and $b \neq 0$.
Under  a {\it relatively prime part} of $b$ with respect to $a$
written $RP(a, b)$,  we have in mind a factor $t$ of $b$ such that, if
$b =s t$, then
\begin{itemize}
\item[(i)] $(t, a) = 1$;
\item[(ii)] $(s', a) \neq 1$ for any non-unit factor $s'$ of $s$.
\end{itemize}

The element  $s$ (if it exists) is called  an   {\it adequate part} of $b$ with respect to $a$.
A  ring $R$ is called {\it adequate} \cite[p.\,225]{Helmer} if ${RP}(a, b)$ exists for all $a, b \in R$ with $b \neq 0$.
This concept is essentially a formalization of properties of   entire analytic functions rings.
Each commutative principal ideal domain (PID) is  adequate, but the converse is not true, in particular, the ascending
chain condition on ideals
may not be satisfied.  Each adequate ring is an elementary divisor ring  \cite[Theorem 3, p.\,234]{Helmer}. The ring of all continuous real-valued functions defined on a completely regular (Hausdorff) space $X$ is an example of an adequate ring, which is regular and  every prime ideal is maximal \cite[Corollaries  3.6,3.8 p.\,386]{Gillman_Henriksen_2}.  Each local ring as well as each commutative  von Neumann regular ring is adequate \cite[Theorem 11, p.\,365]{Gillman_Henriksen}. Adequate rings with zero-divisors in their Jacobson radical were investigated   by Kaplansky  \cite[Theorem 5.3,  p.\,473]{Kaplansky}.
Note that not every elementary divisor ring is adequate \cite[Corollary 6.7, p.\,386]{Gillman_Henriksen_2} and  in an adequate domain each nonzero prime ideal is contained in a unique maximal ideal \cite[Corollary 6.6, p.\,386]{Gillman_Henriksen_2}.
B\'ezout rings in which each regular element is adequate were  investigated in  \cite{Zabavsky_Gatalevych_2}.   Moreover, generalized adequate rings were introduced in~\cite{Zabavskii},  forming a new class of elementary divisor rings that includes adequate rings as a subclass.

Gatalevych~\cite{Gatalevich} was the first to attempt applying the notion of adequacy to noncommutative rings.
He introduced a new concept of adequacy for noncommutative rings and proved that a generalized right adequate (in the sense of Gatalevych) duo Bezout domain is an elementary divisor domain~\cite[Theorem~2, p.\,117]{Gatalevich}.
In the present article, we propose an alternative definition of adequate rings, which differs from the one introduced by Gatalevych~\cite[Definition~1, p.\,116]{Gatalevich}.
Using an example in \S4, we demonstrate certain advantages of our definition.
Our  definition of the adequacy of a ring is the following:

Let $K$ be a B\'ezout (not necessarily commutative) ring with  $1\not=0$.   An element $0 \neq b\in K$  is called  {\it left adequate} to  $a\in K$ if
either
$aR+bR=R$
or,  if  $aR+bR\neq R$ then
there exists $s $ such that $b=st$ and the  following conditions hold:
\begin{itemize}

\item[(i)] $s'K+aK\ne K$ for each $s'\in K$ such that $sK\subset s'K\ne K$;

\item[(ii)]   for each  $t'\in K$ such that $tK\subset t'K\ne K$  there exists a decomposition $st'=pq$ such that $pK+aK=K$.
\end{itemize}

An element $s$ is called a {\it left adequate part} of $b$ with respect to $a$.
The right adequate part of $b$ with respect to $a$ is defined by analogy.

A subset $A \subseteq K$ is called \textit{left} (respectively, \textit{right}) \textit{adequate} if each of its nonzero elements is left (respectively, right) adequate to all elements of $A$.
If each nonzero element of $A$ is both left and right adequate to all elements of $A$, then the set $A$ is called \textit{adequate}.

It is easy to see that if $K$ is a commutative PID, then our definition coincides with the one given by Helmer \cite[p.\,225]{Helmer}.

\smallskip
Our first main result is the following:

\begin{theorem} \label{T:1}
Let  $R$ be a commutative PID  such that $1\not=0$.  The set of  nonsingular $ 2\times 2$ matrices over $R$ is an adequate set.
\end{theorem}

Let $R$ be a  commutative PID with $1\not=0$.
A subset of  $R^{*} \setminus \{1\} $ consisting of all indecomposable
divisors of an  element $a\in R$ is   called  the {\it  spectrum} of  $a$ and is denoted by $\Sigma (a)$.
The spectrum  of a nonsingular matrix $A\in R^{2\times 2}$   is the set $\Sigma (A):=\Sigma (\alpha_{2})$ (see \eqref{E:1}). Matrices $M, N\in {R}^{2\times 2}$ are  called   {\it strongly right associated} if there is a matrix $U\in {\GL}_2(R)$ such that $M=NU$.

Let $A,B,C,D,A_1,B_1 \in {R}^{2\times 2}$.
If $A=BC$,  then $A$ is called a {\it right multiple} of $B$.  If $A=DA_{1} $ and $B=DB_{1} $, then $D$ is called a {\it left common  divisor} of   $A$ and $B$. In addition, if   $D$ is a right multiple of each left common   divisor of $A$ and $B$, then   $D$ is called a {\it  left  greatest   common  divisor}  of   $A$ and $B$, which  we denoted  by  $D:=(A,B)_{l}$. The left  greatest   common  divisor $(A,B)_{l}$ is unique up to right strongly  associates \cite[Theorem 1.12, p.\,39]{Mon}.

Let $A \in {R}^{2\times 2}$. In view of equation  \eqref{E:1}, we  use  the following presentation:
\begin{equation}\label{E:2}
A:=P_{A}^{-1}\cdot \diag (\alpha_{1},  \alpha_{2} )\cdot  Q_{A}^{-1},
\end{equation}
in which   $\diag (\alpha_{1},  \alpha_{2})=\mathrm{SNF}(A)$, and
$P_{A}, P_{B}$ are the left and right  transforming matrices of $A$.

Our next main result is the following:
\begin{theorem}\label{T:2}
Let  $R$  be a  commutative PID such that $1\not=0$ and let
\[
A:=P_{A}^{-1} \cdot \diag (\alpha_{1},  \alpha_{2} ) \cdot Q_{A}^{-1}\quad   \text{and} \quad  S:=P_{S}^{-1} \cdot \diag (\sigma_{1},  \sigma_{2} ) \cdot Q_{S}^{-1}
\]
be nonsingular matrices  of the form \eqref{E:2}.
Each left divisor of the  matrix   $S$  has a nontrivial left common divisor with the matrix  $A$ if and only if
$\Sigma (\sigma_{i} ) \subseteq \Sigma (\alpha_{i} )$ for  $i=1,2$ and  one of the following conditions holds:
\begin{itemize}
\item[(i)]  $\Sigma (\sigma_{2} )\subseteq \Sigma (\alpha_{1} )$;

\item[(ii)]
\[
P_{S} = \left[\begin{array}{cc}  {m_{11} } & {m_{12} } \\ {q_{1} \cdots q_{k} m_{21} } & {m_{22} } \end{array}\right] P_{A}, \qquad  \qquad (m_{ij}\in R)
\]
where $\{q_{1},\ldots,  q_{k} \}= \Sigma (\sigma_{2} ) \backslash \Sigma (\alpha_{1})$.
\end{itemize}
\end{theorem}


\section{Preliminaries, lemmas and proofs}

For each $2 \times 2 $ nonsingular matrices $A,B$ of the form \eqref{E:2}  we define the matrix
$[\tau_{ij}]:= P_{B} P_{A}^{-1}$ and the set
\begin{equation}\label{E:3}
\boldsymbol{\mathrm{L}}_{\alpha_1, \beta_2}:=
\left\{
\left[\begin{array}{cc}  {l_{11} } & {l_{12} } \\ {\frac{\beta_{2} }{(\beta_{2},  \alpha_{1} )} l_{12} } & {l_{22} } \end{array}\right]\in \GL_2(R)\mid l_{ij}\in R
\right\}.
\end{equation}

In the sequel we will  use the following facts:

\begin{fact} \label{F:1}  \cite[Theorem 1, p.\,851]{UMZh2015}
Let $R$ be a  commutative   elementary divisor ring  and let  $A,B\in {R}^{2\times 2}$ of the form \eqref{E:2}. Then
\begin{itemize}
\item[(i)]  ${\rm SNF}((A,B)_{l}) =
\diag\big((\alpha_{1}, \beta_{1} ), (\alpha_{2}, \beta_{2}, [\alpha_{1}, \beta_{1} ]\tau_{21} )\big);
$
\item[(ii)] $A,B$ are  left relatively prime (i.e., $(A,B)_{l}=I$)  if and only if
\[
(\alpha_{2}, \beta_{2}, [\alpha_{1}, \beta_{1} ]\tau_{21} )=1.
\]
\end{itemize}
\end{fact}

\begin{fact} \label{F:2} \cite[Theorem 4.3, p.\,127]{Mon} Let $R$ be a  commutative  elementary divisor ring  and let  $A,B\in {R}^{2\times 2}$ of the form \eqref{E:2}.
The matrix $B$ is a left divisor of $A$ (i.e.,  $A=BC$) if and only if $\beta_{i} |\alpha_{i} $ for  $i=1,2$ and $P_{B} =LP_{A} $, in which $L\in \boldsymbol{\mathrm{L}}_{\alpha_1, \beta_2}$ (see \eqref{E:3}).
\end{fact}

\begin{fact} \label{F:3}    \cite[Theorem 4.4, p.\,128]{Mon}  Let $R$ be a  commutative  elementary divisor ring   and let  $A\in {R}^{2\times 2}$ of the form \eqref{E:2}. Let $\beta_1,\beta_2\in R$ such that $\beta_1 |\beta_2$ and $\beta_{i} |\alpha_{i} $ for  $i=1,2$.
The set of all left divisors  of $A$ with the Smith form  $\diag (\beta_{1},  \beta_{2} )$ has the form
\[
\left(\boldsymbol{\mathrm{L}}_{\alpha_1,  \beta_2}P_{A} \right)^{-1} \cdot \diag (\beta_{1},  \beta_{2} )\cdot \GL_{2} (R).
\]

\end{fact}

\begin{lemma}\label{LLM:1}
Let  $R$  be a  commutative B\'ezout domain and let  $A, B \in R^{n \times n}$ (with $n\geq 2$) be  nonsingular matrices.  If $\det(B)$  is indecomposable in $R$ and $(A, B)_l\neq I$ then $A=BC$.
\end{lemma}

\begin{proof} Let  $D:=(A, B)_l \neq I$. Clearly,  $B=DB_1$ and  $\det(D) | \det(B)$.  Thus   $\det(D)$   and $\det(B)$ are strong associates in  $R$, i.e., $\det(B) = \det(D) e$ for some $e \in U(R)$.  Consequently, $\det(B_1)=e$, so  $B_1 \in \GL_n(R)$ and    $D=BB_1^{-1}$.   Since $A=DA_1$,  we have $A=BB_1^{-1}A_1=BC$,  where $C=B_1^{-1}A_1$.
\end{proof}

\begin{proof}[Proof of Theorem \ref{T:2}.  Necessity. ] Let   $\omega \in \Sigma (\sigma_{1} )$. Thus
$\sigma_{1} =\omega \sigma_{1}'$ and   $\sigma_{2} =\omega \sigma_{2}'$ for some $\sigma_{1}', \sigma_{2}'\in R$. If
$M:=P^{-1} \cdot  \diag({1},  {\omega}) \cdot  Q^{-1}$ and
\[
M_1:=\left(Q \cdot  \diag({\omega}, 1)\cdot  P\right)\left(P_{S}^{-1} \cdot  \diag({\sigma_{1}'},   {\sigma_{2} '}) \cdot  Q_{S}^{-1}\right),
\]
in which $P,Q$ are arbitrary  invertible matrices, then
\[
\begin{split}
S=P_{S}^{-1} \cdot  \diag({\sigma_{1}},   {\sigma_{2}})  \cdot  Q_{S}^{-1}=M \cdot M_1
\end{split}
\]
and  $(A, M)_{l} \ne I$.
Taking into account that $\det(M)$ is indecomposable in $R$, we obtain that
$A=MA_{1}$  by Lemma \ref{LLM:1}. Consequently, all matrices  $L$ with  $\mathrm{SNF}(L)=\diag (1, \omega )$ are left divisors of $A$. In accordance with     \cite[Theorem  5.3 p.\,152 and Property 4.11 p.\,147]{Mon} we have $\omega |\alpha_{1}$ and   $\Sigma (\sigma_{1} )\subseteq \Sigma (\alpha_{1} )$.
\smallskip

\noindent
\underline{Case 1.} Suppose that   $ \Sigma (\sigma_{2} )\subseteq \Sigma (\alpha_{1})$.
Reasoning  similarly as before, we obtain that every matrix with  Smith normal form $\diag(\sigma_{1}, \sigma_{2})$ has a nontrivial left common divisor with  $A$.
\smallskip

\noindent
\underline{Case 2.} Let  $\mu \in \Sigma (\sigma_{2} )\setminus \Sigma (\alpha_{2})$. Thus   $\sigma_{2} =\mu \cdot \mu_{1} $ and  $(\mu,  \alpha_{2} )=1$. If   $C:=P_{S}^{-1}\cdot   \diag(1,  {\mu})$ and $C_1:=\diag(\sigma_{1},  {\mu_{1}}) \cdot   Q_{S}^{-1}$, then  $S=C C_{1}$.
Since $(\det(C),  \det(A))=1$, we have   $(A, C)_{l} =I$, a contradiction. Consequently  $\Sigma (\sigma_{2} )\subseteq \Sigma (\alpha_{2} )$.

Let  $\Sigma (\sigma_{2} )\backslash \Sigma (\alpha_{1} )=\{ q_{1},  \ldots,   q_{k} \} $ for $k\geq 1$ and let $i\in\{1,\ldots, k\}$. Thus  $\sigma_{2} =q_{i} \delta_{i}$ for some $\delta_{i}\in R$. If
$D:=P_{S}^{-1}\cdot   \diag(1,  {q_{i}})$ and $D_1:=\diag(\sigma_{1},  {\delta_{i}}) \cdot   Q_{S}^{-1}$, then
\[
S=P_{S}^{-1} \cdot  \diag({\sigma_{1}}, {\sigma_{2}})\cdot   Q_{S}^{-1}=D\cdot  D_1.
\]
All  left divisors $L$  of $S$  (including  $D$) with  $\mathrm{SNF}(L)=\diag(1, q_{i})$ belong to the set
\begin{align*}
{\bf W}&=\left\{\left( \boldsymbol{\mathrm{L}}_{\sigma_1, q_i}  P_{S} \right)^{-1} \cdot  \diag({1},  {q_{i}}) \cdot  \GL_{2} (R)\right\} & \text{by }& \text{Fact  \ref{F:3}}\\
&=\left\{ \left(\left[\begin{array}{cc}  {l_{11} } & {l_{12} } \\ {q_{i} l_{12} } & {l_{22} } \end{array}\right]P_S\right )^{-1}\cdot   \diag({1},  {q_{i}}) \cdot  \GL_{2} (R)\mid l_{ij}\in R \right\} & \text{since  }&  (q_{i},  \sigma_{1} )=1.
\end{align*}
Let us fix  $M:=P_{M}^{-1}  \cdot  \diag({1},  {q_{i}})\cdot   Q_{M}^{-1}\in {\bf W}$,  in which  $P_{M}: = \left[\begin{array}{cc}  {h_{11} } & {h_{12} } \\ {q_{i} h_{21} } & {h_{22} } \end{array}\right]  P_{S}$ for some $h_{pl}\in R$ and   $Q_{M}\in \GL_2(R)$ is fixed. The matrix  $M$ is a left divisor of $S$, so   $(A, M)_{l} \ne I$. Hence,   $d_{i}:=(\alpha_{2},  q_{i}, \alpha_{1} \tau_{21}^{(i)} ) \ne 1$ (see Fact  \ref{F:1}(ii)), where
\begin{equation} \label{E:4}
[\tau_{mn}^{(i)}]  :=P_{M} P_{A}^{-1} = \left[\begin{array}{cc}  {h_{11} } & {h_{12} } \\ {q_{i} h_{21} } & {h_{22} } \end{array}\right]   (P_{S} P_{A}^{-1}).
\end{equation}
Since $d_i \mid q_i$ and both $d_i$ and $q_i$ are indecomposable elements of $R$, it follows that they are strongly associated.  Taking into account that $d_i, q_i \in R^{\ast}$, we obtain
\[d_{i}=
q_{i}  =(\alpha_{2},  q_{i},  \alpha_{1} \tau_{21}^{(i)} )=(\alpha_{2},  (q_{i},  \alpha_{1} \tau_{21}^{(i)} ))=(\alpha_{2},  q_{i},  \tau_{21}^{(i)} ),
\]
so  $q_{i} | \tau_{21}^{(i)}$, i.e.,  $\tau_{21}^{(i)} =q_{i} n_i$  for some $n_i\in R$. It is obvious (see  \eqref{E:4}) that
\begin{equation} \label{E:5}
P_{S} P_{A}^{-1}=\left[\begin{array}{cc}  {h_{11} } & {h_{12} } \\ {q_{i} h_{21} } & {h_{22} } \end{array}\right]  ^{-1} \left[\begin{array}{cc}  {\tau_{11}^{(i)} } & {\tau_{12}^{(i)} } \\ {q_{i} n_{i} } & {\tau_{22}^{(i)} } \end{array}\right]=\left[\begin{array}{cc}  {p_{11} } & {p_{12} } \\ {q_{i} p_{21} } & {p_{22} } \end{array}\right]
\qquad (p_{mn}\in R).
\end{equation}
Let us show that \eqref{E:5} holds independently of the choices of $P_{S} $ and  $P_{A}$.
Let $P'_{S}$ and $P'_{A}$ be arbitrary left transforming matrices of $S$ and $A$, respectively. By \cite[Property 2.2 p.\,61]{Mon}),
\[
P'_{S} =FP_{S} \quad\text{and}\quad  P'_{A} =TP_{A},
\]
where
\[F:=  \left[\begin{array}{cc}  {f_{11} } & {f_{12} } \\ {\frac{\sigma_{2} }{\sigma_{1} } f_{21} } & {f_{22} } \end{array}\right], \;\; T^{-1}:=  \left[\begin{array}{cc}  {t_{11} } & {t_{12} } \\ {\frac{\alpha_{2} }{\alpha_{1} } t_{21} } & {t_{22}} \end{array}\right]\in \GL_2(R) \qquad (f_{mn}, t_{mn}\in R).
\]
Thus
\[
\begin{split}
P'_{S} (P'_{A} )^{-1} &=F(P_{S} P_{A}^{-1}) T^{-1} \\
&=  \left[\begin{array}{cc}  {f_{11} } & {f_{12} } \\ {\frac{\sigma_{2} }{\sigma_{1} } f_{21} } & {f_{22} } \end{array}\right]  \cdot  \left[\begin{array}{cc}  {p_{11} } & {p_{12} } \\ {q_{i} p_{21} } & {p_{22} } \end{array}\right]  \cdot  \left[\begin{array}{cc}  {t_{11} } & {t_{12} } \\ {\frac{\alpha_{2} }{\alpha_{1} } t_{21} } & {t_{22} } \end{array}\right].
\end{split}
\]
Since $q_{i}\in \Sigma(\sigma_2)\setminus \Sigma(\alpha_1)$ and  $\Sigma(\sigma_1) \subset \Sigma(\alpha_1),$ then   $q_{i}\in \Sigma(\sigma_2)\setminus \Sigma(\sigma_1)$. Hence, $q_{i}\in \Sigma\left(\frac{\sigma_{2} }{\sigma_{1} } \right)$. As
$q_{i}\notin \Sigma(\alpha_1)$ therefore
$q_{i}\in \Sigma\left(\frac{\alpha_{2} }{\alpha_{1} } \right)$ and $q_{i}\in \Sigma\left(\frac{\sigma_{2} }{\sigma_{1} } \right)\cap  \Sigma \left(\frac{\alpha_{2} }{\alpha_{1} } \right)$. Therefore
\[
P'_{S}(P'_{A})^{-1} = \left[\begin{array}{cc}  {p'_{11} } & {p'_{12} } \\ {q_{i} p'_{21} } & {p'_{22} } \end{array}\right]   \qquad\qquad   (p'_{mn}\in R).
\]
Consequently, \eqref{E:5} holds regardless of the choice of $P_{S}$ and $P_{A}$.

Now we need to proceed in the same way with the remaining elements of the set $\{ q_{1}, \ldots, q_{k} \}$.
As a result, the matrix $P_{S}$ takes the form described in Theorem~\ref{T:2}(ii).
\bigskip

 {\it Sufficiency.}
Let  $S=LM$, in which the nontrivial divisor  $L:=P_L^{-1}\cdot  \diag(\lambda_1, \lambda_2)\cdot Q_L^{-1}$ has the form \eqref{E:2} .
\bigskip

\noindent
\underline{Case 1.} If  $\Sigma (\sigma_{2} )\subseteq \Sigma (\alpha_{1} )$, then  $\Sigma (\lambda_{2} )\subseteq \Sigma (\alpha_{1} )$ by  Fact \ref{F:2}.
This yields   $(\alpha_{2}, \lambda_{2}, \alpha_{1})\neq 1$, so $(A, L)_l \neq I$ by  Fact \ref{F:1}(ii) for arbitrary $P_S\in {\GL}_2(R)$.
\bigskip

\noindent
\underline{Case 2.}
Let $\Sigma (\sigma_{2} ) \subseteq \Sigma (\alpha_{1})  \cup \{q_{1}, \ldots , q_{k} \}$  for   $k \geq 1$ and each $q_{i}\not\in \Sigma (\alpha_{1})$. If
$1\neq \gamma \in \Sigma (\lambda_{2})  \cap\Sigma (\alpha_{1})$, then
$L=L_1 L_2$, where
\[
\textstyle
L_1:=P_L^{-1}\cdot  \diag(1, \gamma)
\quad \text{and}\quad
L_2:=\diag \left(\lambda_1, \frac{\lambda_2}{\gamma}\right)\cdot Q_L^{-1}.
\]
According to the above considerations, $L_1$ is a left divisor of $A$.

Using  $(\alpha_{2}, \gamma, \alpha_{1})\neq 1$ and Fact \ref{F:1}(ii) we have  $(A, L_1)_l \neq I$. The element $\det(L_1)$ is indecomposable in $R$, so  $A= L_1A_1$ by Lemma \ref{LLM:1} and  $(A, L)_l \neq I$.

Suppose   $\delta \in \{ q_{1},  \ldots,  q_{k} \}\cap  \Sigma(\lambda_2)$.
It is easy to see that $L=F_1 F_2$, where
\[
F_1:=P_L^{-1}\cdot  \diag(1, \delta)
\qquad \text{and}\qquad
\textstyle F_2:=\diag \left(\lambda_1, \frac{\lambda_2}{\delta}\right)\cdot Q_L^{-1}.
\]
The set of all left divisors of $S$ with Smith normal form $\diag(1, \delta)$ (see Fact~\ref{F:2}) is given by
\[
{\bf W}:=\left\{\left( \boldsymbol{\mathrm{L}}_{\sigma_1, \delta}  P_{S} \right)^{-1} \cdot  \diag({1},  \delta) \cdot  \GL_{2} (R)\right\}.
\]
Since $(\delta, \sigma_1)=1$, any matrix $D\in {\bf W}$ can be written in the form
$D=P_{D}^{-1}\cdot  \diag (1, \delta)\cdot  Q_{D}^{-1}$,
where $P_{D}=\left[\begin{array}{cc}  {l_{11} } & {l_{12} } \\ {\delta  l_{12} } & {l_{22} } \end{array}\right] P_{S}$, and
$Q_{D} \in \GL_2(R)$. Consequently,   we have
\[
\begin{split}
P_{D} P_{A}^{-1}
&= \left[\begin{array}{cc}   {l_{11} } & {l_{12} } \\ {\delta  l_{12} } & {l_{22} } \end{array}\right] P_{S} P_{A}^{-1}\\
&= \left[\begin{array}{cc}  {l_{11} } & {l_{12} } \\ {\delta  l_{12} } & {l_{22}}\end{array}\right]
\cdot
\left[\begin{array}{cc}   {m_{11} } & {m_{12} } \\ {q_{1} \cdots q_{k} m_{21} } & {m_{22} } \end{array}\right]=
\left[\begin{array}{cc}   {l'_{11} } & {l'_{12} } \\ {\delta  l'_{12} } & {l'_{22} } \end{array}\right],
\end{split}
\]
so   $P_{D} = \begin{pmatrix}  {l'_{11} } & {l'_{12} } \\ {\delta  l'_{12} } & {l'_{22} } \end{pmatrix} P_{A}$.
Therefore $A=DA_{2}$ by Fact  \ref{F:2}.   It follows that each left  divisor $D$ of $S$ with $\mathrm{SNF}(D)=\diag(1, q_i)$ for $i=1, \ldots, k$  (including $L_1$)
is a left divisor of the  matrix  $A$ too. Consequently  $F_1:=P_L^{-1}\cdot  \diag(1, \delta)$ is a  left divisor of $A$. It means that $(A,L)_l\not=I$.
\end{proof}

Let $A$ and $B$ be nonsingular matrices. We study the properties and structure of the left divisors of $B$ that have a nontrivial left common divisor with $A$.

\begin{lemma}\label{LLM:2}
Let  $R$  be a  commutative   PID and  let $A,  S, T$ be nonsingular matrices in $R^{2\times 2}$. If all left divisors of $S$ have  a common left divisor with $A$, then
\[
\Sigma (S)\subseteq \Sigma \big((A, ST)_{l} \big).
\]
\end{lemma}

\begin{proof} Let $ST:=P_{ST}^{-1} \cdot \diag (\beta_{1},  \beta_{2}) \cdot Q_{ST}^{-1}$ and
$S=P_{S}^{-1} \cdot \diag (\sigma_{1},  \sigma_{2}) \cdot Q_{S}^{-1}$ have form \eqref{E:2}.
Let  $ \mu \in \Sigma (S)$. Thus  $\sigma_{2} =\mu \sigma_{2} '$ and $S=S_{1} S_{2}$, where
\[
\textstyle
S_1:=P_{S}^{-1} \cdot \diag (1, \mu) \qquad\text{and}\qquad   S_{2}:=\diag (\sigma_{1},  \sigma_{2} ') \cdot  Q_{S}^{-1}.
\]
By assumption, $(A, S_{1})_{l} \ne I$. Since $\det(S_{1})$ is an indecomposable element of $R$, it follows from Lemma~\ref{LLM:1} that $S_{1}$ is a left divisor of $A$.
Hence, $S_{1}$ is a left common divisor of the matrices $A$ and $ ST$, and thus a left divisor of $(A, ST)_{l}$.
Consequently, $\mu = \Sigma(S_{1}) \subseteq \Sigma\big((A, ST)_{l}\big)$, and therefore $\Sigma(S) \subseteq \Sigma\big((A, ST)_{l}\big)$.
\end{proof}

\begin{lemma}\label{LLM:3}
Let  $R$  be a  commutative   PID and let $A,B,S\in {R}^{2\times 2}$ be nonsingular matrices of the  form \eqref{E:2}:
\[
\begin{split}
A:&=P_{A}^{-1}  \cdot  \diag (\alpha_{1},  \alpha_{2} ) \cdot  Q_{A}^{-1},\qquad  B:=P_{B}^{-1} \cdot  \diag (\beta_{1},  \beta_{2} ) \cdot  Q_{B}^{-1},\quad \\
S:&=  P_{S}^{-1} \cdot  \diag (\sigma_{1},  \sigma_{2} ) \cdot  Q_{S}^{-1},\qquad \text{and} \quad [\tau_{ij}]:=P_{B} P_{A}^{-1}.
\end{split}
\]
Each left divisor of the matrix $S$ has a nontrivial left common divisor with  $A$ and $B=ST$ if and only if $S$ satisfies the conditions of Theorem  \ref{T:2} and
\begin{equation} \label{E:6}
\left(
\textstyle
\frac{\sigma_{2} }{(\sigma_{2}, \beta_{1} )}, q_{1} \cdots q_{k}
 \right)| \tau_{21},
\end{equation}
where $\sigma_{2} =q_{1}^{r_1} \cdots q_{k}^{r_k} d_{2}$  for  $q_{1}, \ldots , q_{k} \in
\Sigma (\sigma_{2})  \setminus \Sigma (\alpha_{1})$, $r_i \in \mathbb{N} \cup \{0\}$, $i= 1, \ldots , k,$
and  \linebreak  $\Sigma (d_{2} )\subseteq  \Sigma (\alpha_{1})$.
\end{lemma}

\begin{proof} Necessity.  Since $S$ is a left divisor of $B$,  $\Sigma (\sigma_i )\subseteq  \Sigma (\beta_i)$
for $i=1, 2$ and $P_{S}=L  P_{B}$, where
\[
L:=\left[\begin{array}{cc}  {l_{11} } & {l_{12} } \\ {\frac{\sigma_{2} }{(\sigma_{2}, \beta_{1} )} l_{21} } & {l_{22} } \end{array}\right]  \qquad (l_{ij}\in R)
\]
by  Fact  \ref{F:2}.  Each left divisor of $S$ has a left common divisor with $A$, so $S$ satisfies the conditions of  Theorem  \ref{T:2}. Hence
 $P_{S} =NP_{A}$, where
\[
N:=
\left[\begin{array}{cc}  {n_{11} } & {n_{12} } \\ {q_{1} \cdots q_{k} n_{21} } & {n_{22} } \end{array}\right]
  \qquad(n_{ij}\in R).
 \]
Consequently, $P_{S} =NP_{A}=L  P_{B}$.
It follows that
\[
\begin{split}
[\tau_{ij} ]=P_{B} P_{A}^{-1}&=L^{-1} N =\underbrace{ \left[\begin{array}{cc}  {l'_{11} } & {l'_{12} } \\ {\frac{\sigma_{2} }{(\sigma_{2}, \beta_{1} )} l'_{21} } & {l'_{22} } \end{array}\right]  }_{L^{-1} }   \left[\begin{array}{cc}  {n_{11} } & {n_{12} } \\ {q_{1} \cdots q_{k} n_{21} } & {n_{22} } \end{array}\right]\\
&= \left[\begin{array}{cc}  {m_{11} } & {m_{12} } \\ {\left(\frac{\sigma_{2} }{(\sigma_{2}, \beta_{1} )},  q_{1} \cdots q_{k} \right)m_{21} } & {m_{22} } \end{array}\right] \qquad( l'_{ij}, m_{ij}\in R).
\end{split}
\]
Therefore, the condition   \eqref{E:6} is fulfilled.

{\it  Sufficiency.}  There  exist  invertible matrices (see  \cite[Lemma 5.10, p.\,193]{Mon})
\[
C^{-1}:=\left[\begin{array}{cc}  {c_{11} } & {c_{12} } \\ {\frac{\sigma_{2} }{(\sigma_{2}, \beta_{1} )} c_{21} } & {c_{22} } \end{array}\right]\quad \text{and}   \quad  D:=\left[\begin{array}{cc}  {d_{11} } & {d_{12} } \\ {q_{1} \cdots q_{k} d_{21} } & {d_{22} } \end{array}\right]
\]
such that $P_{B} P_{A}^{-1}=C^{-1} D$.  The matrix $S:=(CP_{B} )^{-1}\cdot  \diag (\sigma_{1},  \sigma_{2} )$ is a left divisor of  $B$ by Fact  \ref{F:2}. Moreover,  each left divisor of  $S=(DP_{A})^{-1}\cdot  \diag (\sigma_{1},  \sigma_{2} )$ has a nontrivial left common divisor with  $A$ by Theorem \ref{T:2}.

Let us show that  \eqref{E:6} holds independently of the choices of $P_{B}, P_{A}\in \GL_2(R)$.
Indeed, if  we choose a different  ordered pair $(P'_{B}, P'_{A})\not=(P_{B}, P_{A})$, then  $P'_{B} =HP_{B}$ and $P'_{A} =TP_{A}$ by \cite[Property 2.2, p.\,61]{Mon},  where
\[
H:=  \left[\begin{array}{cc}  {h_{11} } & {h_{12} } \\ {\frac{\beta_{2} }{\beta_{1} } h_{21} } & {h_{22} } \end{array}\right]\quad  \text{and}\quad T^{-1}:=  \left[\begin{array}{cc}  {t_{11} } & {t_{12} } \\ {\frac{\alpha_{2} }{\alpha_{1} } t_{21} } & {t_{22} } \end{array}\right]\quad( h_{ij}, t_{ij}\in R).
\]
Thus
\[
\begin{split}
[\tau'_{ij} ]:=P'_{B} (P'_{A} )^{-1} &=HP_{B} P_{A}^{-1} T^{-1} =H[\tau_{ij} ]T^{-1}\\
&=  \left[\begin{array}{cc}  {h_{11} } & {h_{12} } \\ {\frac{\beta_{2} }{\beta_{1} } h_{21} } & {h_{22} } \end{array}\right]  \cdot  \left[\begin{array}{cc}  {\tau_{11} } & {\tau_{12} } \\ {\tau_{21} } & {\tau_{22} } \end{array}\right]  \cdot { \left[\begin{array}{cc}  {t_{11} } & {t_{12}} \\ {\frac{\alpha_{2} }{\alpha_{1} } t_{21}} & {t_{22}} \end{array}\right]}.
\end{split}
\]
Hence
\[
\begin{split}
\tau'_{21} = \tau_{21}(h_{22} t_{11}) +\textstyle\frac{\beta_{2} }{\beta_{1} }(h_{21} \tau_{11}t_{11}+\frac{\alpha_{2} }{\alpha_{1} }h_{21} \tau_{12}t_{21}) +\frac{\alpha_{2} }{\alpha_{1}}(h_{22} \tau_{22} t_{21}).
\end{split}
\]
Obviously, $\frac{\beta_{2} (\sigma_{2}, \beta_{1} )}{\beta_{1} \sigma_{2} } =\frac{(\beta_{2} \sigma_{2}, \beta_{2} \beta_{1} )}{\beta_{1} \sigma_{2} } \in R$, so   $\frac{\sigma_{2} }{(\sigma_{2}, \beta_{1} )} |\frac{\beta_{2} }{\beta_{1} }$.
Taking into account that  $q_{1}, \ldots,   q_{k} \in \Sigma (\alpha_{2} )$ and  $(q_{1} \cdots q_{k},  \alpha_{1} )=1$, we obtain that
$
\textstyle(q_{1} \cdots q_{k} )|\frac{\alpha_{2} }{\alpha_{1}}$,
and
$\left(\frac{\sigma_{2} }{(\sigma_{2}, \beta_{1} )},  q_{1} \cdots q_{k} \right)|\left(\frac{\beta_{2} }{\beta_{1} }, \frac{\alpha_{2} }{\alpha_{1} } \right)$.
Consequently,   $\left(\frac{\sigma_{2} }{(\sigma_{2}, \beta_{1} )},  q_{1} \cdots q_{k} \right)|\tau'_{21}$.
\end{proof}

\begin{proof}[\bf Proof of Theorem \ref{T:1}]  If $(A,B)_{l}=I$ then $ B$ is adequate to $A$.

Let  $(A,B)_{l}\neq I$, where  $A:=P_{A}^{-1}\cdot  \diag (\alpha_{1},  \alpha_{2})\cdot  Q_{A}^{-1}$ and   $B=P_{B}^{-1} \cdot \diag (\beta_{1},  \beta_{2}) \cdot  Q_{B}^{-1}$ have  the form \eqref{E:2}.
Set $\mathrm{SNF}\big((A,B)_{l}\big):=\diag (\omega_{1},  \omega_{2} )$ and   $[\tau_{ij}]:=P_{B} P_{A}^{-1}$.

 Due to  Lemma \ref{LLM:2},   if $D$ is a left divisor of $B$ and none of its left divisors is relatively prime to $A$, then
$\Sigma(D)\subseteq \Sigma\big((A,B)_{l}\big)$. By Fact \ref{F:1}(i),   $\Sigma (\omega_{i} )\subseteq \Sigma (\alpha_{i} )$ for $i=1, 2$.  Set
\[
\begin{split}
\Sigma (\omega_{1} )& =\Sigma(( \alpha_{1}, \beta_1) ):=\{ p_{1}, \ldots, p_{m} \};\\ \Sigma (\omega_{2} ):&=\{ p_{1}, \ldots, p_{n} \} \cup \{ q_{1}, \ldots, q_{l} \} \cup \{ q_{l+1}, \ldots, q_{k} \},
\end{split}
\]
where
\[
\{
p_{1}, \ldots,  p_{n} \} \subseteq \Sigma (\alpha_{1} ), \quad n\geq m,  \quad q_{i}\notin  \Sigma (\alpha_{1}), \quad i=1, \ldots , k,  \]
\[
 \{ q_{1}, \ldots, q_{l} \} \subseteq \Sigma (\tau_{21} ), \qquad \{ q_{l+1}, \ldots, q_{k} \} \cap \Sigma (\tau_{21} )=\varnothing.
\]
By Fact \ref{F:1}(i), we have $\omega_{i} \mid \beta_{i}$ for $i = 1, 2$, so we can write
\begin{equation}\label{E:7}
\beta_{1}=\underbrace{\left(p_{1}^{r_{1} } \cdots p_{m}^{r_{m} } \right)}_{\sigma_{1} } \cdot \left(q_{1}^{u_{1} } \cdots q_{l}^{u_{l} } \right)\cdot \left(q_{l+1}^{u_{l+1} } \cdots q_{k}^{u_{k} } \right)\cdot d=\sigma_{1} \cdot \beta '_{1}, \quad u_i \in \mathbb{N}\cup \{0\},
\end{equation}
\begin{equation}\label{E:8}
\begin{split}
\beta_{2} &=\underbrace{\left(\left(p_{1}^{r'_{1} } \cdots p_{m}^{r'_{m} } \right)\cdot \left(p_{m+1}^{r_{m+1} } \cdots p_{n}^{r_{n} } \right)\cdot \left(q_{1}^{u'_{1} } \cdots q_{l}^{u'_{l} } \right)\cdot \left(q_{l+1}^{u_{1+1} } \cdots q_{k}^{u_{k} } \right)\right)}_{\sigma_{2} } \beta '_{2}\\
& =\sigma_{2} \cdot \beta '_{2},
\end{split}
\end{equation}
where  $(d, \alpha_{2} )=1$,\quad
$(\beta '_{2}, p_{1} \cdots p_{n} \cdot q_{1} \cdots q_{l} )=1$,  \quad  $r'_{i} \ge r_{i}$,
for  $i=1,\ldots, m$ and $ u'_{j} \ge u_{j}\geq 0$  for  $j=1,\ldots, l$.
It follows that
\[
\textstyle
\frac{\sigma_{2} }{(\sigma_{2}, \beta_{1} )} =\left(p_{1}^{r'_{1} -r_{1} } \cdots p_{m}^{r'_{m} -r_{m} } \right)\cdot \textstyle\left(p_{m+1}^{r_{m+1} } \cdots p_{n}^{r_{n} } \right)\cdot \left(q_{1}^{u'_{1} -u_{1} } \cdots q_{l}^{u'_{l} -u_{l} } \right).\]
Since $q_{1}, \ldots, q_{l} \in \Sigma (\tau_{21} )$,
\[
\textstyle
\left(\frac{\sigma_{2} }{(\sigma_{2}, \beta_{1} )}, q_{1} \cdots q_{k} \right)=\left(q_{1}^{u'_{1} -u_{1} } \cdots q_{l}^{u'_{l} -u_{l} }, q_{1} \cdots q_{k} \right)|\tau_{21}.\]
According to   \cite[Lemma 5.10, p.\,193]{Mon},  we can write
\begin{equation} \label{E:9}
\textstyle P_{B} P_{A}^{-1} = \left[\begin{array}{cc}  {f_{11} } & {f_{12} } \\ {\frac{\sigma_{2} }{(\sigma_{2}, \beta_{1} )} f_{21} } & {f_{22} } \end{array}\right]    \left[\begin{array}{cc}  {l_{11} } & {l_{12} } \\ {q_{1} \cdots q_{k} l_{21} } & {l_{22} } \end{array}\right]
\quad ( f_{ij}, l_{ij}\in R)
.
\end{equation}
Let us consider the  matrix
\begin{equation}\label{Matrix_S_1}
S:=\left( \left[\begin{array}{cc}  {f_{11} } & {f_{12} } \\ {\frac{\sigma_{2} }{(\sigma_{2}, \beta_{1} )} f_{21} } & {f_{22} } \end{array}\right]  ^{-1} P_{B} \right)^{-1}  \left[\begin{array}{cc}  {\sigma_{1} } & {0} \\ {0} & {\sigma_{2} } \end{array}\right].
\end{equation}
Using Fact \ref{F:2},  $S$ is the  left divisor of $B$, i.e.  $B=ST$ for some $T\in R^{2\times 2}$.
From  \eqref{E:9}, we have
\[
\left[\begin{array}{cc}  {f_{11} } & {f_{12} } \\ {\frac{\sigma_{2} }{(\sigma_{2}, \beta_{1} )} f_{21} } & {f_{22} } \end{array}\right]^{-1}   P_{B} = \left[\begin{array}{cc}  {l_{11} } & {l_{12} } \\ {q_{1} \cdots q_{k} l_{21} } & {l_{22} } \end{array}\right]  P_{A}.
\]
It follows that the matrix $S$   can also be written in  the following form:
\begin{equation}\label{Matrix_S_2}
S=\left( \left[\begin{array}{cc}  {l_{11} } & {l_{12} } \\ {q_{1} \cdots q_{k} l_{21} } & {l_{22} } \end{array}\right]  P_{A} \right)^{-1}  \left[\begin{array}{cc}  {\sigma_{1} } & {0} \\ {0} & {\sigma_{2} } \end{array}\right].
\end{equation}
Consequently, each left divisor of $S$ has a nontrivial common left divisor with $A$ by Theorem~\ref{T:2}. Therefore, $S$ satisfies part~(i) of the definition of an adequate part of $B$ with respect to $A$.
\smallskip

Assume that $T = T_{1} T_{2}$ is a decomposition of $T$ into a product of two of its nontrivial divisors. Let us consider the following two cases:
\bigskip

\noindent
\underline{Case 1.} Let  $\Sigma(ST_1) \not\subseteq  \Sigma ((A,B)_{l}) \neq \varnothing$. Hence, there exists $t\in \Sigma(ST_1)\setminus  \Sigma ((A,B)_{l})$. It means that  $ST_{1}$ has a left divisor $L$ with $\mathrm{SNF}(L)=\diag(1, t)$ such that  $(A,L)_{l} =I$ (by the same trick as the one used in the proof of  Lemma \ref{LLM:2}).
\bigskip

\noindent
\underline{Case 2.} Let  $\Sigma (ST_{1} )\subseteq \Sigma ((A,B)_{l})$ and
$\mathrm{SNF}(ST_{1} )=\diag (\mu_1, \mu_2)$.
Based on the construction of the elements $\sigma_1$ and $\sigma_2$, it follows that   $\det(ST_{1})$ has the divisor  $q_{i}^{u_{i} +1} $ in which  $l+1\le i\le k$.
\bigskip

\noindent
\underline{Case 2a.}   Let $q_{i} | \mu_1$. Any matrix with the Smith normal form $\diag (q_{i}, q_{i})$ is a left divisor of $ST_{1}$ by \cite[Theorem  5.3 p.\,152 and Property 4.11 p.\,147]{Mon}.
Consider  the matrix $M:=P_M^{-1}\diag (q_{i}, q_{i})$, where $P_M:= \left[\begin{array}{cc}  0 & 1\\ 1 & 0  \end{array}\right]P_A$. It is obvious that    $M=M_1M_2$, where  $M_1:=P_M^{-1}\diag (1, q_{i})$ and  $M_2:=\diag (q_{i}, 1)$. Since $(\alpha_1, q_{i})=1$ and $P_MP_A^{-1}= \left[\begin{array}{cc}  0 & 1\\ 1 & 0  \end{array}\right]$,  we have $(A,M_1)_{l}=I$ by  Fact \ref{F:1} (ii).
Thus $ST_{1}=MN=M_1(M_2N)$ for some $N$.
\bigskip

\noindent
\underline{Case 2b.}   Let $(q_{i}, \mu_1)=1$. Clearly  $q_{i}^{u_{i} +1}| \mu_2$.
The matrix $K := P_{ST_1}^{-1} \diag(1, q_i)$ is a left divisor of $ST_1$, and therefore also a left divisor of the matrix $B$.
Since
$
\frac{q_{i}^{u_{i} +1} }{(q_{i}^{u_{i} +1},   \beta_{1} )} =q_{i}$, we have
 ${P_{ST_{1}}:= \left[\begin{array}{cc}  {k_{11} } & {k_{12} } \\ {q_{i} k_{21} } & {k_{22} } \end{array}\right]  P_{B}}$ by Fact \ref{F:2}.
Thus
\[
\begin{split}
[\tau'_{ij}]:=P_{ST_{1}} P_{A}^{-1} &= \left[\begin{array}{cc}  {k_{11} } & {k_{12} } \\ {q_{i} k_{21} } & {k_{22} } \end{array}\right] (P_{B} P_{A}^{-1}) = \left[\begin{array}{cc}  {k_{11} } & {k_{12} } \\ {q_{i} k_{21} } & {k_{22} } \end{array}\right]   \left[\begin{array}{cc}  {\tau_{11} } & {\tau_{12} } \\ {\tau_{21} } & {\tau_{22} } \end{array}\right]  \\
&= \left[\begin{array}{cc}  {*} & {*} \\ {q_{i} k_{21} \tau_{11} +k_{22} \tau_{21} } & {*} \end{array}\right].
\end{split}
\]
The matrix $\left[\begin{array}{cc}  {k_{11}} & {k_{12}} \\ {q_{i} k_{21}} & {k_{22}} \end{array}\right]$ is invertible.
Hence, $(q_{i}, k_{22}) = 1$.
By assumption, $(q_{i}, \tau_{21}) = 1$, so $(q_{i}, \tau_{21}') = 1$.
Since $q_{i} \not\in \Sigma(\alpha_1)$, we have $(q_{i}, \alpha_{1}) = 1$, so
$(\alpha_{2}, q_{i}^{u_{i} + 1}, \alpha_{1} \tau_{21}') = 1$.
Consequently, $(A, K)_{l} = I$ by Fact~\ref{F:1}(ii).
This means that $S$ satisfies part~(ii) of the definition of an adequate part of $B$ with respect to $A$, so the set of nonsingular $2 \times 2$ matrices over $R$ is a left adequate set.
Applying the transpose operator, we obtain that this set is also a right adequate set.
Consequently, the set of all nonsingular $2 \times 2$ matrices over $R$ is an adequate set.
\end{proof}

\bigskip

\section{Some examples}

We now present an algorithm for constructing an adequate part of a matrix in $R^{2 \times 2}$.

\textbf{Example 1}. Let $R$ be a PID and  let $a,b,c, f,m, n\in R\setminus\{U(R)\cup \{0\}\}$ be  pairwise relatively prime indecomposable elements. Let
\[
\begin{split}
A&:= \diag({ab}, \;{ab^{2} cfm}),  \quad   B:= \left[\begin{array}{cc}  {1} & {0} \\ {-f} & {1} \end{array}\right]    \diag({b^{2} c}, \;{ab^{3} c^{2} fn}),\\
P_{A}&=I,\qquad P_{B}=  \left[\begin{array}{cc}  {1} & {0} \\ {f} & {1} \end{array}\right],  \quad
 [\tau_{ij}]:= P_{B} P_{A}^{-1} =P_{B}= \left[\begin{array}{cc}  {1} & {0} \\ {f} & {1} \end{array}\right].
\end{split}
\]
Clearly,   $\Sigma (A)=\{ a, b, c, f, m\}$, $\Sigma (B)= \{a, b, c, f, n \}$ and  $\mathrm{SNF}\big((A,B)_{l}\big)=\diag({b}, {ab^{2} cf})$ by Fact \ref{F:1}(i).
Using the notation of  Theorem \ref{T:1} we have  that $q_{1} q_{2} =cf$.
An adequate  part of  $B$ with respect to  $A$ (see Theorem \ref{T:1})  has the following Smith normal form  $
\diag ({b^{2}}, {ab^{3} cf}):=\diag({\sigma_{1}}, {\sigma_{2}})$.
Note that
\[
\textstyle
\left(\frac{\sigma_{2} }{(\sigma_{2}, \beta_{1} )},  q_{1} q_{2} \right)=(abf, cf)=f|\tau_{21}.
\]
It is easy to check that
\[P_{B} P_{A}^{-1} = \left[\begin{array}{cc}  {1} & {0} \\ {f} & {1} \end{array}\right]  = \left[\begin{array}{cc}  {1} & {0} \\ {abfy} & {1} \end{array}\right]   \cdot  \left[\begin{array}{cc}  {1} & {0} \\ {cfx} & {1} \end{array}\right]  \]
in which $cx+aby=1$. It follows that
\[
\left[\begin{array}{cc}  {1} & {0} \\ {-abfy} & {1} \end{array}\right]  P_{B} = \left[\begin{array}{cc}  {1} & {0} \\ {cfx} & {1} \end{array}\right]   P_{A}.
\]
Consequently, an adequate  part of  $B$ with respect to  $A$ has the following form:
\[
\begin{split}
S:&=\left( \left[\begin{array}{cc}  {1} & {0} \\ {-abfy} & {1} \end{array}\right]    \left[\begin{array}{cc}  {1} & {0} \\ {f} & {1} \end{array}\right]   \right)^{-1}  \left[\begin{array}{cc}  {b^{2} } & {0} \\ {0} & {ab^{3} cf} \end{array}\right]  \;\\
&= \left[\begin{array}{cc}  {1} & {0} \\ {f(aby-1)} & {1} \end{array}\right]    \left[\begin{array}{cc}  {b^{2} } & {0} \\ {0} & {ab^{3} cf} \end{array}\right] = \left[\begin{array}{cc}  b^{2} & {0} \\ {b^{2}f(aby-1)} & {ab^{3} cf}  \end{array}\right]
\end{split}\]
by Theorem \ref{T:1}.
In this case
$B=ST,$
where $T= \left[\begin{array}{cc}  {c} & {0} \\ {-y} & {cn} \end{array}\right]$. {\hfill \rm  $\diamond$}
\bigskip

Each commutative PID $R$ is adequate in the sense of Helmer, as noted in the Introduction.
It is easy to verify that the adequate and the relatively prime parts of an element $b \in R$ with respect to $a \in R$ are defined up to strong associates.
However, this statement does not hold in the case of the ring $R^{2 \times 2}$, as shown in the next example:
\bigskip

\textbf{Example 2}. Let  $R=\mathbb{Z}$ be the ring of integers.
 Let
$$
A:=\diag(\alpha_1, \alpha_2)= \diag(2,\;\;  2\cdot3\cdot5\cdot7),\quad B:= \left[\begin{array}{cc}  {1} & {0} \\ -3 & 1 \end{array}\right]\cdot \diag(2\cdot 3^2\cdot5^2, \;\;  2^2\cdot 3^3\cdot5^4).
$$
Then
$$
P_A=I,\quad P_B=\left[\begin{array}{cc}  {1} & {0} \\ 3 & 1 \end{array}\right],\quad
   [\tau_{ij}]:=P_B P^{-1}_A=P_B,
$$
$$
\Sigma(\alpha_1)=\{2\}, \quad \Sigma(\omega_2)=\{2, 3,5\}, \quad \Sigma(\tau_{21})=\{1, 3\},\quad \{5\}\cap \Sigma(\tau_{21})=\varnothing.
$$
According to Fact \ref{F:1}(i), $\mathrm{SNF}((A, B)_{l})=\diag(\omega_1, \omega_2)=\diag(2, 2\cdot3\cdot5)$.
The left adequate part  of $B$ with respect to $A$ has the following Smith normal form $\Phi:=\diag( 2, 2^2\cdot 3^3\cdot5^2)$ (see the proof of Theorem \ref{T:1}).
The matrices
\[S:=\left[\begin{array}{cc}  {1} & {0} \\ -3\cdot5 & 1 \end{array}\right]\cdot\Phi\quad\text{and}\quad
S_1:=\left[\begin{array}{cc}  {1} & {0} \\ 3\cdot5 & 1 \end{array}\right]\cdot\Phi
\]
are left divisors of the matrix $B$:
\[
B=S\left[\begin{array}{cc}  3^2\cdot5^2 & {0} \\ 2 & 5^2 \end{array}\right]=S_1\left[\begin{array}{cc}  3^2\cdot5^2 & {0} \\ -3 & 5^2  \end{array}\right],
\]
and are also adequate parts of $B$ with respect to   $A$ by Theorem \ref{T:2}.
However  (see \cite[Theorem 4.5, p.\,128]{Mon}) the  matrices $S$ and $S_1$ are not right   strong  associates. {\hfill \rm  $\diamond$}

\bigskip

Let $S$ be an adequate part of $B$ with respect to $A$ with the presentation  \eqref{Matrix_S_1}.  Example~2 shows that if $S'$ is another adequate part of $B$ with respect to $A$, then $S$ and $S'$ are not necessarily right associated. Based on this example, we put forward the following.

\begin{hypothesis}
The adequate part of $B$ with respect to $A$ is defined up to equivalence.
\end{hypothesis}

\bigskip

\section{Adequate  rings in the sense of Gatalevych}
Gatalevych \cite[Definition 1, p.\,116]{Gatalevich} proposed the following definition for noncommutative B\'ezout rings which was  already indicated in the Introduction.

Let $K$ be  a B\'ezout ring and let  $a\in K$. An  element $b\in K$ is called  {\it left adequate in the sense of Gatalevych}  to $a\in K$  if the following conditions hold:
\begin{itemize}
 \item[(i)]  there exist  elements $s,t\in K$ such that $b=st$ and  $tK  +aK=K$;
 \item[(ii)] $s'K  +aK  \ne K$ for each  $s'\in K\setminus U(K)$ such that $sK \subset s'K\ne K$.
 \end{itemize}

The shortcomings of this  definition are demonstrated by  the next example:

\textbf{Example 3}. Let  $R$ be a commutative PID, and let $a,d,c\in R\setminus\{U(R)\cup \{0\}\}$ be  pairwise relatively prime indecomposable elements. Let
\[
\begin{split}
A:&= \diag({a}, {a^{2} dc}),\quad P_A=I,\quad   B:= \left[\begin{array}{cc}  {1} & {0} \\ {d} & {d^{3} c^{2} } \end{array}\right]  = \left[\begin{array}{cc}  {1} & {0} \\ {d} & {1} \end{array}\right]   \left[\begin{array}{cc}  {1} & {0} \\ {0} & {d^{3} c^{2} } \end{array}\right] , \quad   P_{B}= \left[\begin{array}{cc}  {1} & {0} \\ {-d} & {1} \end{array}\right],\\
A_1:&=\diag({a}, {a^{2} c}),\quad \qquad T:=\left[\begin{array}{cc}  {1} & {0} \\ {1} & {1} \end{array}\right]\cdot \diag(1,{d^{2} c^{2}}), \qquad  S:=\diag(1,d).
\end{split}
\]
It is easy to check that  $A=SA_{1}$ and  $B=ST$. Since $(A,T)_{l} =I$  (see Fact \ref{F:1}(ii)), the decomposition  $B=ST$ satisfies   the definition of Gatalevych.

On the other hand,
\[
B=S_{1} T_{1}=\big(P^{-1}_{S_{1}} \cdot \diag(1,d^{3}) \cdot Q^{-1}_{S_{1}}\big)
\cdot
\big(P^{-1}_{T_{1}} \cdot \diag(1,c^{2}) \cdot Q^{-1}_{T_{1}}\big),
\]
where
\[
\begin{split}
S_1&=\left[\begin{array}{cc}  {1} & {0} \\ d^{3}+{d} & {d^{3} } \end{array}\right], \quad
P_{S_{1} } = \left[\begin{array}{cc}  {1} & {0} \\ -{d} & {1} \end{array}\right],\quad  Q_{S_{1} } = \left[\begin{array}{cc}  {1} & {0} \\ -{1} & {1} \end{array}\right],\\
T_1&= \left[\begin{array}{cc}  {1} & {0} \\ -1 & {c^{2} } \end{array}\right],\qquad\;
P_{T_{1}}= \left[\begin{array}{cc}  {1} & {0} \\ 1 & {1} \end{array}\right],\qquad\;  Q_{T_{1} } =I.
\end{split}
\]
Each left divisor of  $S_{1}$ has   a  nontrivial left common divisor with   $A$ by Theorem  \ref{T:2} and  $(A,T_{1} )_{l} =I$ by Fact \ref{F:1}(ii),  so the decomposition $B=S_{1} T_{1}$ also satisfies Gatalevych's   definition.
However,  $S$ is the left divisor of $S_{1}$, because  $S_{1}  = S  \left[\begin{array}{cc}  {1} & {0} \\ 1& d^2 \end{array}\right]$.

It should be noted that the decompositions $B = ST$ and $B = S_1T_1$ also exhibit another undesirable property.
Let us consider the cosets $S \GL_2(R)$ and $S_1 \GL_2(R)$, i.e., the sets of all right strongly associated matrices to the matrices $S$ and $S_1$, respectively.
According to Fact~\ref{F:1}(ii), each left divisor of the matrices from $S\; \GL_2(R)$ and $S_1\; \GL_2(R)$ has a nontrivial left common divisor with the matrix $A$.
However, if $U, V \in \GL_2(R)$ and
\[
B = (SU)(U^{-1}T) = (S_1V)(V^{-1}T_1),
\]
then it does not necessarily follow that $(A, U^{-1}T)_{l} = I$ and $(A, V^{-1}T_1)_{l} = I$.
Indeed, if
\[
U:= \left[\begin{array}{cc}  {1} & {0} \\ 1-d & {1} \end{array}\right]
\quad\text{and}\quad  V:= \left[\begin{array}{cc}  {1} & {0} \\ -1 & {1} \end{array}\right],
\]
then $T':=U^{-1}T= \left[\begin{array}{cc}  {1} & {0} \\ d & {1} \end{array}\right]\cdot \diag(1,d^{2}c^2)$ and  $T_1':=V^{-1}T_1=  \diag(1, c^2)$. It is easy to see that  $(A, T')_{l} \neq I$ and $(A, T_1')_{l} \neq I$. {\hfill \rm  $\diamond$}

\end{document}